\documentclass[letterpaper,11pt]{amsart}

\usepackage{amscd,amssymb,amsfonts,amsmath,amsthm,color}
\usepackage{pdfsync}

\usepackage{color}
\definecolor{gray}{rgb}{0.3,0.3,0.3}

\theoremstyle{plain}
\newtheorem{theorem}{Theorem}[section]
\newtheorem{lemma}[theorem]{Lemma}
\newtheorem{proposition}[theorem]{Proposition}
\newtheorem{corollary}[theorem]{Corollary}

\theoremstyle{definition}
\newtheorem{definition}[theorem]{Definition}
\newtheorem{remark}[theorem]{Remark}

\newcommand{\nc}{\newcommand}
\newcommand{\rc}{\renewcommand}
\nc{\mc}{\mathcal}
\rc{\t}{\text}
\nc{\loccit}{\emph{loc. cit. }}
\nc\pf{\noindent Proof: }
\nc{\Hom}{\t{Hom}}
\nc{\tot}{\t{tot}}
\nc{\dual}{^{\vee}}
\nc{\op}{^{\t{op}}}
\nc{\coh}{\t{coh}}
\nc{\iso}{\cong}
\rc{\d}{\operatorname{d}}
\nc{\Id}{\operatorname{Id}}
\nc{\dgmod}{\operatorname{dg-mod}}

\nc{\compose}{\circ}
\nc{\sheafsym}{\mathcal{S}\t{ym}}
\nc{\rend}{\operatorname{REnd}}
\nc{\rhom}{\operatorname{RHom}}
\nc{\sheafrend}{\mathcal{R}\mc{E}\t{\emph{nd}}}
\nc{\sheafrhom}{\mathcal{R}\mc{H}\t{\emph{om}}}
\nc{\sheafhom}{\mathcal{H}\t{om}}
\nc{\exterior}{{\textstyle\bigwedge\nolimits}}
\nc{\ex}{\exterior}
\nc{\cok}{\operatorname{Coker}}
\rc{\ker}{\operatorname{Ker}}
\nc{\Lotimes}{{\overset{L}{\otimes}}}
\rc{\to}{\rightarrow}
\nc{\ot}{\leftarrow}
\nc\xto[1]{\xrightarrow{#1}}
\nc{\too}{\longrightarrow}
\nc{\oot}{\longleftarrow}
\nc{\into}{\hookrightarrow}
\nc{\mapsinto}{\hookrightarrow}
\nc{\D}{\operatorname{D}}
\nc{\Dsg}{\D_{{sg}}}
\nc{\Db}{\D^{{b}}}
\nc{\Dbgr}{\Db_{{gr}}}
\nc{\Dgr}{\D_{{gr}}}
\nc{\Dsggr}{\Dsg^{{gr}}}
\nc{\Cgr}{\operatorname{C}_{{gr}}}
\nc{\cCgr}{\cC_{gr}}
\nc{\cDgr}{\cD_{gr}}
\nc{\cDsggr}{\cD_{gr}^{sg}}
\nc{\cDbgr}{\cD_{gr}^{b}}
\nc{\cDsg}{\cD_{sg}}
\nc{\cDb}{\cD^{b}}
\rc{\H}{\operatorname{H}}

\nc{\cA}{\mc{A}}\nc{\cB}{\mc{B}}\nc{\cC}{\mc{C}}\nc{\cD}{\mc{D}}\nc{\cE}{\mc{E}}\nc{\cF}{\mc{F}}\nc{\cG}{\mc{G}}\nc{\cH}{\mc{H}}\nc{\cI}{\mc{I}}\nc{\cJ}{\mc{J}}\nc{\cK}{\mc{K}}\nc{\cL}{\mc{L}}\nc{\cM}{\mc{M}}\nc{\cN}{\mc{N}}\nc{\cO}{\mc{O}}\nc{\cP}{\mc{P}}\nc{\cQ}{\mc{Q}}\nc{\cR}{\mc{R}}\nc{\cS}{\mc{S}}\nc{\cT}{\mc{T}}\nc{\cU}{\mc{U}}\nc{\cV}{\mc{V}}\nc{\cW}{\mc{W}}\nc{\cX}{\mc{X}}\nc{\cY}{\mc{Y}}\nc{\cZ}{\mc{Z}}
\nc{\PP}{\mathbb{P}}
\nc{\CC}{\mathbb{C}}
\nc{\ZZ}{\mathbb{Z}}
\nc{\NN}{\mathbb{N}}
\nc{\QQ}{\mathbb{Q}}
\rc{\AA}{\mathbb{A}}

\let\oldmarginpar\marginpar
\renewcommand\marginpar[1]{\-\oldmarginpar[\raggedleft\footnotesize #1]%
{\raggedright\footnotesize #1}}
\nc\note[1]{\marginpar{#1}}

\title{Equivalence of the Derived Category of a Variety with a Singularity Category}
\author{M. Umut Isik}

\begin{document}
\begin{abstract}
  We prove an equivalence between the derived category of a variety and the equivariant/graded singularity category of a corresponding singular variety. The equivalence also holds at the dg level. 
\end{abstract}

\maketitle

Let $Y$ be a smooth variety over a field $k$.  
An object of consideration is the bounded derived category $\Db(Y)$ of coherent sheaves on $Y$.
For a singular variety $Z$, one can also consider the quotient of $\Db(Z)$ by the full subcategory of perfect complexes. This quotient, denoted $\Dsg(Z)$ is called the singularity category. When $Z$ is smooth with enough locally free sheaves, all bounded complexes of coherent sheaves are quasi-isomorphic to perfect complexes, so $\Dsg(Z)$ is trivial in this case. In the presence of a $k^\times$-action or induced grading on the structure sheaf, one also considers a $k^\times$-equivariant or graded version of the singularity category, namely the categories $\Dsg^{k^\times}(Z)$ or $\Dsggr(Z)$.  

The main result of this paper is an equivalence between the derived category of any variety which is the zero scheme of a section of a vector bundle on a smooth variety, and the $k^\times$-equivariant/graded singularity category of a corresponding singular variety $Z$ which we now describe.

The singular variety $Z$ depends on the expression of $Y$ as the zero scheme of a regular section $s\in \H^0(X, \cE)$ of a vector bundle $E$ with sheaf of sections $\cE$ on a smooth variety $X$. Let $\pi: E \to X$ be the projection map.
Let $W: E \dual \to k$ be the function on the total space of the dual vector bundle $E \dual$ induced by the $s$, i.e. the pairing of the pullback of $s$ to $E \dual$ with the tautological section of the pullback $\pi^*\cE\dual$. 
Let $Z = W^{-1}(0) \subset E \dual$ be the zero locus of this function. $Z$ has a dilation action of $k^\times$ coming from the one on $E \dual$. 
We prove that there is an equivalence of triangulated categories:
\begin{equation}\nonumber
  \Db(Y) \iso \Dsg^{k^\times}(Z)\text{.}
\end{equation}

By the results of V. Lunts and D. Orlov on uniquess of enhancements of triangulated categories to differential graded (dg) categories \cite{uniquenessofenhancements}, this equivalence induces a quasi-equivalence between dg enhancements of the triangulated categories above. In particular, there is a quasi-equivalence 
\begin{equation*}
  \cD^{{b}}(Y) = \cD^{k^\times}_{{sg}}(Z)
\end{equation*}
between the dg category of complexes of coherent sheaves over $Y$ and the $k^\times$-equivariant dg singularity category of Z.

In topological string theory, the bounded derived category of coherent sheaves $\Db(Y)$ apprears as the category of B-branes in a nonlinear sigma model with target Y. The $\CC^\times$-equivariant singularity category appears (via the equivalence to the category of matrix factorizations, which is proven in the affine case in \cite{orlov1}) as the category of B-branes in a Landau-Ginzburg theory.  
In this language, our main result can be expressed as a correspondence between D-branes of type B in a non-linear sigma model with target $Y$ and D-branes of type B in a Landau-Ginzburg model $(E\dual, W)$. It had been conjectured that this correspondence existed as these two theories are related by renormalization group flow. A similar result should also hold for the categories of A-branes, namely the derived Fukaya category of $Y$ and the Fukaya-Seidel category of $(E\dual, W)$.

Another motivation for this work is the philosophy of derived noncommutative geometry. The idea of this approach to geometry is to replace all geometric constructions on a space $X$ by constructions on the dg category of sheaves of modules on $X$ in order to allow one to do geometry with dg-categories that do not come from a space. (see, for example, \cite{KaKoP}) The equivalence above is therefore an equivalence between these two objects considered as two non-commutative spaces in this sense. 

We would like to mention some relations with other works.
The construction of the singular variety $Z$ is similar to the construction in \cite{orlov2} where it is proven (Theorem 2.1 \emph{loc. cit.}) that $\Dsg(\PP(Z)) \iso \Dsg(Y)$. It is not immediately clear, however, whether the equivalence we construct gives the same functor as \emph{loc. cit.} after taking the appropriate Verdier quotients. 

In the work of V. Baranovsky and J. Pecharich \cite{baranovskypecharich}, our equivalence is used in an application of a theorem on how Fourier-Mukai equivalences of DM stacks over $\AA^1$ give equivalences of the singularity categories of the singular fibres; the application provides a generalization of a theorem of Orlov \cite{orlov3} on the derived categories of Calabi-Yau hypersurfaces in weighted projective spaces to products of Calabi-Yau hypersurfaces in simplicial toric varieties with nef anticanonical class. Our argument was previously sketched in \emph{loc. cit.}. It appeared at the time, that taking the split completion of the singularity category was necessary to prove the equivalence but this turned out not to be the case. 

In an upcoming paper, I. Shipman proves, in the line bundle case, what corresponds to our equivalence in the setting of global matrix factorizations, and gives a new proof of a theorem of Orlov on Calabi-Yau hypersurfaces in projective space. \cite{orlov3} 

\subsection*{Acknowledgements} I am grateful to my advisor Tony Pantev for his guidance, support and ideas throughout my studies. I would like to thank Vladimir Baranovsky for his help in this project. I also want to thank Dima Arinkin, Dragos Deliu, Tobias Dyckerhoff, Bernhard Keller, Sasha Kuznetsov and Pranav Pandit for useful discussions.

\subsection*{Summary} We now describe the proof of our equivalence and give a summary of contents. The main part of the proof is given in section 3. 
The proof involves sheaves of graded dg algebras and modules over them, for which we give definitions and set up notation in section 1. The grading on the algebras and modules is to keep track of the $k^\times$-action. We have $\Db(Z) \iso \Db(\pi_* \cO_Z)$, where $\pi: Z \to X$ is the projection. We first replace the pushforward of $\cO_Z$ to $X$ by a resolution 
$$\cB = \sheafsym\cE \oplus \varepsilon \sheafsym\cE$$
which is a sheaf of graded dg algebras on $X$ with differential $\d\varepsilon=s$. Since $\pi_*(\cO_Z)$ and $\cB$ are quasi-isomorphic, we have $\Db(Z) \iso \Db(\cB)$

We apply a Koszul duality statement to $\cB$, called linear Koszul duality, developed by Mirkovic and Riche \cite{linkos}. Linear Koszul duality is an equivalence between the symmetric algebra of a dg vector bundle and the symmetric algebra of the shifted dual dg vector bundle. We explain this setup in section 2. Applying this gives an equivalence between the derived categories of coherent graded dg modules over $\cB$ and over its quadratic dual sheaf of dg algebras 
$$\cA = \exterior^\bullet \cE \otimes \cO_X[t]$$
with Koszul type differential. 
We then show that the duality takes perfect objects to objects which are supported on $X$, hence inducing an equivalence between the quotients by these subcategories. 

Finally, again in section 3, we show taking the quotient of the derived category of coherent modules over $\cA$ by the full subcategory of modules supported on $X$ has the effect of formally inverting the $t$ in $\cA$, which is similar to restricting to the complement of the zero-section. The sheaf of graded dg algebras $\cA[t^{-1}]$ obtained this way is a copy of a shift of the Koszul resolution of $Y$ in $X$ in each degree, and its derived category of coherent dg modules  
is equivalent to the bounded derived category of coherent sheaves on $Y$.

\section{Sheaves of DG-Algebras and Modules}
We work over an algebraically closed field $k$ of characteristic 0. We will work with sheaves of dg algebras and dg modules. 
A pair $(X,\cA)$ where $X$ is an ordinary scheme and $\cA$ is a sheaf of dg algebras with $\cO_X$-linear differential, with non-positive grading, with $\cA^0=\cO_X$, and whose cohomological graded pieces are quasi-coherent over $\cO_X$ is known as a dg scheme \cite{cfkap}. We will work with such pairs $(X,\cA)$ with slightly different assumptions on the grading.

There is an additional internal grading that we will consider on our dg algebras and modules. Geometrically, this grading corresponds to a $k^\times$-action. All the sheaves of dg algebras and dg modules we consider will have the internal grading and we will not always reflect the existence of this internal grading in our notation.  

We refer to \cite{richethesis}, \cite{linkos} for a more complete treatment of sheaves of dg algebras and modules. Here we only give some definitions in order to make the article as self-contained as possible.

A sheaf of graded dg modules $\cM$ over the pair $(X,\cA)$ is an $\cO_X$-module together with the data giving $\cM(U)$, for each open set $U\subset X$,  the structure of a graded dg $\cA(U)$-module, namely an action:
\begin{equation*}
  \cA(U) \otimes_{\cO_X(U)} \cM(U) \to \cM(U)
\end{equation*}
which is a map of complexes of graded sheaves, 
commuting with the restriction maps. We will refer to these modules as $\cA$-modules without referring to the grading every time. 

A morphism between dg $\cA$-modules $\cM$ and $\cN$ is a collection $\{\phi_U\}$ of maps  
\begin{equation*}
  \phi_U : \cM(U) \to \cN(U)
\end{equation*}
 of internal and cohomological degree $0$, commuting with the restriction maps and the action of $\cA(U)$.   

For a sheaf of graded dg algebras or dg modules $\cF$, we will write $\cF = \bigoplus_{i,j} \cF^i_j$, where $\cF^i_j$ is the piece with cohomological grading $i$ and internal grading $j$.
The operator $[m]$ will denote a shift in the cohomological grading, and the operator $(n)$ will denote a shift in internal grading: 
\begin{equation*}
  \cF[m](n)^i_j = \cF^{i+m}_{j+n}\text{,}
\end{equation*}
with $\d_{\cF[m](n)}= (-1)^m \d_\cF$.

We require that our sheaves of graded dg algebras $\cA$ satisfy that $\cA^0_0=\cO_X$, that each piece $\cA^i_j$ is quasi-coherent, and the piece $\cA_0 = \bigoplus_j \cA^j_0$ is cohomologically non-positively graded.

\begin{definition}
  An $\cA$-module $\cM$ is said to be \emph{quasi-coherent} if each $\cM^i$ is quasi-coherent over $\cO_X$. 
  $\cM$ is said to be \emph{coherent} if it is quasi-coherent and its cohomology sheaf $\H(\cM)$ is coherent over $\H(\cA)$ as a sheaf of graded algebras.
\end{definition}

Quasi-coherent $\cA$-modules and morphisms between them form a $k$-linear category which we will denote by $\Cgr(X,\cA)$ or by $\Cgr(\cA)$. In this category, coherent $\cA$-modules form the full subcategory denoted by $\Cgr^{\t{coh}}(\cA)$. 

$\cA$-modules $\cM$ and $\cN$ are said to be quasi-isomorphic if there is a $\phi \in \Hom_{\Cgr(\cA)}(\cM,\cN)$ that induces isomorphisms $\H(\phi): \H^{\bullet}(\cM) \to \H^{\bullet}(\cN)$. The derived category $\Dgr(X,\cA)$ (or briefly $\Dgr(\cA)$) is defined to be the localization of the homotopy category of $\Cgr(\cA)$ by quasi-isomorphisms. $\Dgr(\cA)$ has the structure of a triangulated category. $\Dbgr(\cA)$ is then the full subcategory in $\Dgr(\cA)$ of coherent $\cA$-modules. 

Alternatively, we can consider the dg category $\cCgr(X,\cA) = \cCgr(\cA)$ of $\cA$-modules by allowing morphisms of internal degree $0$ which do not necessarily have cohomological degree $0$. In this case, each $\Hom_{\cCgr}(\cA)(\cM,\cN)$ is a complex with differential given on a morphism of cohomological degree $d$ by 
\begin{equation*}
  \d(f)=\d_\cN \compose f - (-1)^d f \compose \d_\cM \text{.}
\end{equation*}
The category $\H^0(\cCgr(\cA))$ which is the category obtained by taking $\H^0$ of all the $\Hom$ complexes is the homotopy category of $\Cgr(\cA)$. We can localize the dg category $\cCgr(\cA)$ by all quasi-isomorphisms to obtain the dg derived category $\cDgr(\cA)$. In this case, we have
\begin{equation*}
  \H^0(\cDgr(\cA)) \iso \Dgr(\cA) \text{,}
\end{equation*}
and it is said that $\cDgr(\cA)$ is a dg enhancement of $\Dgr(\cA)$.

We can also consider the full subcategory $\cDbgr(\cA)$ of coherent modules in $\cDgr(\cA)$. We then have
\begin{equation*}
  \H^0(\cDbgr(\cA)) \iso \Dbgr(\cA) \text{.}
\end{equation*}

For a set of objects $\{ S_i \}_{i\in \ZZ}$ in a triangulated category $T$, we denote by $\left< S \right>_{i\in \ZZ}$ the smallest full triangulated subcategory  of $T$ containing all the objects $S_i$ that is closed under direct summands. In particular, this subcategory is closed under finite direct sums and all cones; but not necessarily under infinite direct sums even if they exist in $T$. $\left< S \right>_{i\in \ZZ}$ is said to be the subcategory classically generated by the objects $S_i$. See \cite{bvdb} for more details about this concept.

When we say that a sheaf $\cM$ of $\cA$-modules has a property locally in $\Cgr(X,\cA)$ (respectively $\Dgr(X,\cA)$), we mean that at every point $x\in X$, there is an open immersion $i: U \into X$ such that the property in question holds for $i^*\cM$ as an object of $\Cgr(U, \cA_{|U})$ (respectively $\Dgr(U, \cA_{|U})$).

\begin{definition} An object $\cM$ in $\Dgr(X,\cA)$ is said to be \emph{strictly perfect} if it is an object in the full subcategory $\left< \cA(i) \right>_{i\in \ZZ}$. $\cM$ is said to be \emph{perfect} if it is locally strictly perfect. \end{definition}  

We will denote the full subcategory of perfect $\cA$-modules in $\Dbgr(\cA)$ by $\t{Perf}\cA$.

\begin{definition}
The singularity category of $(X,\cA)$ is defined to be the Verdier quotient:
\begin{equation*}
  \Dsggr(\cA) = {\Dbgr(\cA)}/{\t{Perf}\cA}\text{.}
\end{equation*}
\end{definition}

At the dg level, we can take the dg quotient \cite{drinfeld}, \cite{kellercyclic}, \cite{toen} of the dg derived category $\cDbgr(\cA)$ by the full dg subcategory of perfect $\cA$-modules.  
\begin{definition}
The dg singularity category of $(X,\cA)$ is defined to be the dg quotient
\begin{equation*}
  \cDsggr(\cA) = {\cDbgr(\cA)}/{{\mathcal P}{\t{\emph{erf}}}\cA}\text{.}
\end{equation*}
\end{definition}
We then have
\begin{equation*}
  \H^0(\cDsggr(\cA)) \iso \Dsggr(\cA)\text{.}
\end{equation*}

Let $\cA$ be a sheaf of graded algebras considered as a sheaf of graded dg algebras with trivial differential, with $k^\times$ action on $Z=\mathbf{Spec}\cA$ induced by this grading. Let $\Db_{k^\times}(Z)$ be the bounded derived category of $k^\times$-equivariant coherent sheaves on $Z$ and $\Dsg^{k^\times}(Z)$ be the Verdier quotient of this category by the full subcategory of $k^\times$-equivariant vector bundles on $Z$. This latter subcategory corresponds to the full triangulated subcategory of locally finitely generated locally projective modules over $\cA$. Which is the same as $\left< \cA(i) \right>_{i\in \ZZ}$ since every locally finitely generated locally projective $\cA$-module is locally the direct summand of a free module. Thus we have 
\begin{equation*}
  \Dbgr(\cA) = \Db_{k^\times}(Z)\text{,} \hspace{1cm}\hspace{1cm} \Dsggr(\cA) = \Dsg^{k^\times}(Z) \text{.}
\end{equation*}
Also at the dg level, if we define $\cDb_{k^\times}(Z)$ and $\cDsg^{k^\times}(Z)$ in the same manner, we have
\begin{equation*}
  \cDbgr(\cA) = \cDb_{k^\times}(Z)\text{,} \hspace{2cm} \cDsggr(\cA) = \cDsg^{k^\times}(Z) \text{.}
\end{equation*}
Therefore, our definitions agree with the usual definitions in the case of graded algebras.

We now give our definition of the subcategory of modules supported on $X$. 

\begin{definition}
  An $\cA$-module $\cM$ in $\Dgr(\cA)$ is said to be supported on $X$ if it is locally in $\left< \cO_X(i) \right>_{i\in \ZZ}$. The full subcategory of $\Dbgr(\cA)$ consisting of coherent modules supported on $X$ is denoted by $\Db_{X}(\cA)$.
\end{definition}

In the case discussed above when $\cA$ is a sheaf of graded algebras with trivial differential, $\Db_X(\cA)$ is the subcategory of graded modules which have support on the subscheme $X$ in $Z=\mathbf{Spec}(\cA)$; that is, those objects which are acyclic when pulled back to the open subset $Z \backslash X$. 

If we have a morphism $\varphi: \cA \to \cB$ of non-positively (cohomologically) graded sheaves of graded dg algebras then it induces derived functors
\begin{equation*}
  L\varphi_* : \Dgr(\cA) \to \Dgr(\cB) 
\end{equation*} 
given by the derived tensor product
\begin{equation*}
  \cM \mapsto \cM {\Lotimes}_\cA \cB 
\end{equation*}
and
\begin{equation*}
  R\phi^* : \Dgr(\cB) \to \Dgr(\cA) 
\end{equation*}
given by the restriction of scalars functor. In order to define the derived tensor product above, one needs to use the existence of $K$-flat resolutions. A $K$-flat resolution of a module $\cM$ is a module $\cM'$ with a quasi-isomorphism $\cM' \to \cM$ such that $\cM'$ takes acyclic complexes to acyclic complexes under tensor product. The existence of $K$-flat resolutions in our case is treated in \cite{richethesis} section 1.7. The following proposition shows that the derived category of a sheaf of graded dg algebras only depends on the quasi-isomorphism type.
\begin{proposition}\label{transfer}
  If $\varphi: \cA \to \cB$ is a quasi-isomorphism of sheaves of graded dg algebras which are cohomologically non-positively graded, then the functors $R\varphi^*$ and $L\varphi_*$ are inverse equivalences giving
  \begin{equation*}
    \Dgr(\cA) \iso \Dgr(\cB)\text{.}
  \end{equation*}
Moreover, these functors restrict to an equivalence
  \begin{equation*}
    \Dbgr(\cA) \iso \Dbgr(\cB)\text{.}
  \end{equation*}
\end{proposition}
\begin{proof}
  The first part of the proposition is treated in \cite{richethesis}. For the second part, we need to see that the derived pullback and pushforward functors map coherent modules to coherent modules. The statement is clear for $R\varphi^*$. For $L\varphi_*$, let $\cM$ be a coherent module over $\cA$, which means that $\H(\cM)$ is coherent over $\H(\cA)$. By taking a $K$-flat resolution if necessary, we can assume $\cM$ to be $K$-flat. Since $\cM$ is coherent, we have a two-term resolution 
  \begin{equation*}
    \bigoplus_{\rho=1}^{k_1} \H(\cA)[i_\rho](j_\rho) \too \bigoplus_{\rho=1}^{k_2} \H(\cA)[i'_\rho](j'_\rho) \too \H(\cM) \too 0\text{.}
  \end{equation*}
By picking representatives, we can consider the chain of maps
\begin{equation*}
  \bigoplus_{\rho=1}^{k_1} \cA[i_\rho](j_\rho) \too \bigoplus_{\rho=1}^{k_2} \cA[i'_\rho](j'_\rho) \too \cM \text{.}
\end{equation*}
which induces the above two-term resolution at the level of cohomology. We can take the derived tensor product of these terms with $\cB$. We get a commuting diagram
\begin{equation*}
  \begin{CD}
     \bigoplus_{\rho=1}^{k_1} \cA[i_\rho](j_\rho)  @>>> \bigoplus_{\rho=1}^{k_2} \cA[i'_\rho](j'_\rho)  @>>> \cM \\
     @V{\t{q.i.}}VV  @V{\t{q.i.}}VV @V{\t{q.i.}}VV \\
     \bigoplus_{\rho=1}^{k_1} (\cA \otimes_\cA \cB) [i_\rho](j_\rho)  @>>> \bigoplus_{\rho=1}^{k_2} (\cA \otimes_\cA \cB)[i'_\rho](j'_\rho)  @>>> \cM \otimes_\cA \cB\text{.}\\
  \end{CD}
\end{equation*}
Since $\cM$ is $K$-flat, the derived tensor product is the usual tensor product and we can see by taking cones in $\Dgr(\cA)$ that the vertical morphisms are quasi-isomorphisms. So $\H(\cM \otimes_\cA \cB)$ is coherent over $\H(\cB)$. Thus, by our definition, $R\varphi_* \cM$ is coherent over $\cB$.  
\end{proof}
We will also be able to apply this proposition to some sheaves of dg algebras which do not satisfy the grading assumption by the use of a regrading trick which we discuss in section \ref{proofsection} below. 

\section{Koszul Duality}
In this section, we explain the version of Koszul duality by I. Mirkovic and S. Riche \cite{linkos} for the case of symmetric algebras of dg vector bundles, and describe the constructions and results. It is called linear Koszul duality.

Koszul duality \cite{priddy}, \cite{bgs} is a homological duality phenomenon which generalizes a derived equivalence between an algebra and a corresponding Koszul dual algebra. The main example for us is the equivalence of derived categories between the symmetric and exterior algebras of a vector space, which was used in \cite{bgg} to calculate the bounded derived category of coherent sheaves on projective space. Koszul duality can, in many cases, be shown to express a form of Morita or tilting equivalence between triangulated categories (see for example \cite{beh87}) even though this is often not expressed explicitly.  

Linear Koszul duality \cite{linkos} has the following aspects.
First, the duality is obtained in a relative setting with sheaves of algebras, rather than an algebra over a point. In other words, it is a fiber-wise application of Koszul duality. Second, the duality is obtained for dg algebras and modules. Third, the duality provides a contravariant equivalence rather than the usual covariant one. 

Let $X$ be a scheme. Consider a complex $\cX$ of vector bundles:
\begin{equation*}
  \dots \to 0 \to \cV \xto{f} \cW \to 0\to\dots\text{,}
\end{equation*}
where sections of $\cV$ have cohomological degree $-1$ and internal degree $1$  and sections of $\cW$ have cohomological degree $0$ and internal degree $1$ . For a complex of graded vector bundles (a dg vector bundle) $\cG$ over $\cO_X$, we define the graded symmetric algebra $\sheafsym\cG$ of $\cG$ to be the sheaf tensor algebra of $\cG$ modulo the graded commutation relations $a\otimes b = (-1)^{(\deg_h{a})(\deg_h{b})} b\otimes a$, where $\deg_h$ denotes cohomological degree.

Let $\cB = \sheafsym \cX$ and let $\cA = \sheafsym \cY$ where $\cY$ is the dg vector bundle given by
\begin{equation*}
  \dots \to 0 \to \cW\dual \xto{-f\dual} \cV\dual \to 0 \to \dots\text{,}
\end{equation*}
where sections of $\cV\dual$ are in cohomological degree $1$ and internal degree $-1$ and sections of $\cW\dual$ are in cohomological degree $2$ and in internal degree $-1$ .

The functors that induce the Koszul duality between $\cB$ and $\cA$ are given by 
\begin{equation*}
  \begin{array}{ccccccc}
    F: \Cgr(\cB)\op & \too & \Cgr(\cA) & \hspace{2cm} & G: \Cgr(\cA)\op & \too & \Cgr(\cB) \\
    \cM & \longmapsto & \cA \otimes_{\cO_X} \cM\dual & \, & \cN & \longmapsto & \cB \otimes_{\cO_X} \cN\dual\text{,}
\end{array}
\end{equation*}

\noindent where the  differential of $F(\cM)$ is the sum of two differentials $\d_{F(\cM)} = \d_1 + \d_2$. The first differential $\d_1$ is the product of differentials of $\cM$ and $\cA$, given by 
$$\d_1(a \otimes m) = \d_{\cA}a \otimes m + (-1)^{\deg_hm}a\otimes \d_{\cM\dual}m \text{.}$$
The Koszul type differential $\d_2$ is the sum of Koszul type differentials for $\cV$ and $\cW$. The differential for $\cV$ is given by the composition:
\begin{equation*}
  \cA \otimes_{\cO_{X}} \cM\dual \to \cA \otimes_{\cO_{X}} \cM\dual \to \cA \otimes_{\cO_{X}} \cV \otimes_{\cO_{X}} \cV\dual \otimes_{\cO_{X}} \cM\dual \to \cA\otimes_{\cO_{X}} \cM\dual\text{,}
\end{equation*}
where the first map is a sign adjustment ($a\otimes m \mapsto (-1)^{\deg_h m} a \otimes m$), the second map is the map induced by $\t{id}:\cO_X \to \cV \otimes \cV\dual$ given by the section $\t{id}$ of $\cV\otimes \cV\dual$ and the third map is given by the action of $\cV$ on $\cA$ and of $\cV\dual$ on $\cM\dual$. This corresponds, on a fiber by fiber basis, to multiplication by id$\in V\otimes V^*$ for each fiber $V$ of $\cV$. The differential for $\cW$ is defined by replacing $\cV$ by $\cW$ in this composition. The sum of these two differentials gives $\d_2$. The differential for $G(\cN)$ is given in the same manner.

When restricted to subcategories with the appropriate finiteness conditions, these functors take acyclic complexes to acyclic complexes, hence induce functors between the derived categories. However, these finiteness conditions are different from our coherence condition because the coherent sheaves do not map to each other under the functors. Instead, $\Dbgr(\cB)$ and $\Dbgr(\cA)$ are equivalent to the subcategories of objects satisfying these finiteness conditions. See \cite{linkos}, 3.6.  

\begin{theorem} (\cite{linkos}, Theorems 3.7.1 and 3.6.1)
  The functors $F$ and $G$ induce an exact equivalence
  \begin{equation*}
    \Dbgr(\cB)\op \iso \Dbgr(\cA)
  \end{equation*}
  between the derived categories of coherent dg modules. 
  \label{koszulequivalence}
\end{theorem}

\begin{remark}
  It would be more appropriate for our purposes to use a covariant version of this duality. However, at this stage, it was easier to rely on Mirkovic and Riche's contravariant version. If $\cB$ is Gorenstein in the appropriate sense, then composing with the functor $\sheafrhom_\cB(\bullet, \cB)$ should give a covariant version of this duality. After we use the above theorem in the next section, we use the fact that $Y$ is Gorenstein to turn the equivalence into a covariant one.  
\end{remark}

\section{The Equivalence of the Singularity Category and the Derived Category}\label{proofsection}

We now continue with the notation of the introduction. $X$ is a smooth variety, $Y \subset X$ is given as the zero scheme of a regular section $s \in \H^0(X,\cE)$ where $\cE$ is the sheaf of local sections of the vector bundle $\pi: E \to X$. $W$ is the pairing of the pullback of $s$ to $E$ and the canonical section of the pullback of $E\dual$. Let $Z$ = $W^{-1}(0)$.

The $k^\times$ action on $Z$ induces a grading on sections of $\pi_*\cO_Z$ making it a sheaf of graded algebras. We have $\Dsg^{k^\times}(Z) \iso \Dsggr(\pi_*\cO_Z)$. 

If the sections of $\cE$ are considered in internal degree $1$ and homological degree $0$, then $\sheafsym \cE$ gives the sheaf of algebras of functions on $E\dual$. The exact sequence
\begin{equation}\nonumber
  0 \to \varepsilon \sheafsym \cE \xto{s} \sheafsym \cE \to \pi_*\cO_Z \to 0 
\end{equation}
gives us a resolution of $\pi_*\cO_Z$. Here, $\varepsilon$ is a formal variable in homological degree $-1$ and internal degree $1$. Let $\cB$ be the sheaf of dg algebras 
\begin{equation*}
  \cB = \sheafsym(\dots \to 0 \to \varepsilon\cO_X \xto{s} \cE \to 0 \to \dots)\text{,}
\end{equation*}
where $\varepsilon\cO_X$ is in homological degree $-1$ and internal degree $1$  and $\cE$ is in homological degree $0$ and internal degree $1$. We have that $\varepsilon$ commutes with the other variables and $\varepsilon^2 = 0$ because of the graded commutation relation, so $\cB$, which can also be written as 
\begin{equation}\nonumber
  \cB = \sheafsym\cE \oplus \varepsilon \sheafsym\cE \hspace{1cm} \d_\cB\varepsilon = s
\end{equation}
is the resolution of $\pi_*\cO_Z$ above. So the map $\varphi: \cB \to \pi_*\cO_Z$ which sends $\sheafsym\cE$ to $\pi_*\cO_Z$ and $\varepsilon$ to 0, is a quasi-isomorphism of sheaves of graded dg algebras. 
By Proposition \ref{transfer}, we have an equivalence
\begin{equation*}
  \Dbgr(\cB) = \Dbgr(\pi_*\cO_Z)\text{.}
\end{equation*}
It is clear that, under this equivalence, $\cB$ is taken to $\pi_*\cO_Z$ and vice versa since $\pi_*\cO_Z$ is quasi-isomorphic to $\cB$ as $\cB$-modules, so we also have the induced equivalence of the quotients
\begin{equation*}
  \Dsggr(\cB) \iso \Dsg^{k^\times}(Z)\text{.}
\end{equation*}

We now apply Koszul duality to $\cB$. Let $\cA$ be given by
\begin{equation*}
  \cA = \sheafsym(\dots\to 0 \to \cE\dual \xto{-s\dual} t\cO_X \to 0 \to \dots)\text{,}
\end{equation*}
where $\cE\dual$ are in homological degree $1$ and internal degree $-1$ and $t$ is in homological degree $2$ and internal degree $-1$. So $\cA$ is given by
\begin{equation*}
  \cA = \bigwedge \cE\dual \otimes_{\cO_X} \cO_X[t] \hspace{1cm} \d_\cA f = tf(s)
\end{equation*}
and $\d_\cA$ is $t$-linear. 

By Theorem \ref{koszulequivalence}, we have an equivalence between the derived categories of graded, coherent modules
\begin{equation}\nonumber
  \Dbgr(\cB) \cong \Dbgr(\cA)^{\t{op}}\text{.}
  \label{equivstep1}
\end{equation}

\begin{proposition}\label{quotientequiv}
  The Koszul duality functor $F$ takes perfect modules to modules supported on $X$. The functor $G$ takes modules supported on $X$ to perfect modules. $F$ and $G$ therefore induce an equivalence:
\begin{equation*}
  \Dsggr(\cB) \iso \Dbgr(\cA)/\Db_{X}(\cA)\text{.}
\end{equation*}
\end{proposition}
\begin{proof}
If $U\subset X$ is an open subset, then it is clear from the definition of $F$ that it is defined locally, i.e. that the following diagram strictly commutes at the level of objects \\
\begin{equation*}
  \begin{CD}
    \Dbgr(X,\cB) @>{F}>> \Dbgr(X,\cA) \\
    @V{i^*}VV @V{i^*}VV \\
    \Dbgr(U,\cB_{\mid U}) @>{F}>> \Dbgr(U,\cA_{\mid U}) \text{.}
  \end{CD}
\end{equation*}

It therefore suffices to show that $F_{\mid U}$ takes perfect modules to modules supported on $X$ for affine open sets $U$. We have that $F(\cB(i))=\cO_X(i)$ because $F(\cB)$ is the Koszul resolution of $\cO_X$ (\cite{linkos} section 2.5) so $F$ takes perfect modules to those which are locally in the subcategory classically generated by $\cO_X(i)$ for $i\in \ZZ$. Therefore $F$ takes perfect modules to modules supported on $X$.
Similarly, $G$ is defined locally and takes $\cO_X(i)$ to $\cB(i)$ and therefore takes modules which are supported on $X$ to perfect modules. 
\end{proof}

We now consider the quotient on the right side of the equality of Proposition \ref{quotientequiv}:
\begin{equation*}
  \Dsggr(\cB) \iso \Dbgr(\cA)/\Db_{X}(\cA)\text{.}
  \label{}
\end{equation*}
We will prove that $\Dbgr(\cA)/\Db_{X}(\cA)$ is equivalent to the bounded derived category $\Db(Y)$ of coherent sheaves on $Y$. 

First, we show that taking the quotient by the modules supported on $X$ has the effect of restricting the space to the complement of their support, similar to Serre's original result on the category of coherent sheaves on projective space. Define:
\begin{equation*}
  \cA[t^{-1}] = \bigwedge \cE\dual \otimes_{\cO_X} \cO_X[t,t^{-1}]\text{,} \hspace{1cm} \d_\cA f = tf(s)\text{.}
\end{equation*}

\begin{proposition}\label{propontheway}
  There is an equivalence of triangulated categories
  \begin{equation*}
    \Dbgr(\cA)/\Db_{X}(\cA) \iso \Dbgr(\cA[t^{-1}])\text{.}
  \end{equation*}
\end{proposition}

Before we proceed with the proof of this proposition, we need to understand that modules supported on $X$, i.e. those that are locally in $\left< \cO_X(i) \right>_{i\in \ZZ}$, are the modules which have cohomology coherent as $\cO_X$-modules. We start with the following lemma. 

\begin{lemma}
  Let $X$ be a regular scheme and let $\cR$ be a sheaf of dg algebras such that $\cR^i=0$ for all $i>0$, $\cR_j=0$ for all $j>0$, $\cR^0_0=\cO_X$ and $\H^0(\cR)_0 = \cO_X$. If $\cM$ is a coherent module over $\cR$ and $\H(\cM)$ is coherent when considered as a module over $\cO_X$, then as an object in $\Dbgr(X,\cR)$, $\cM$ is supported on $X$, i.e. it is locally in $\left< \cO_X(i) \right>_{i\in \ZZ}$.
  \label{regradedlemma}
\end{lemma}
\begin{proof}
  Since the question is local, we can consider $X$ to be affine. If $\H(\cM)$ is coherent when considered as an $\cO_X$-module, then the cohomology of $\cM$ is bounded above and below. Furthermore, each $\H^i(\cM)$ is nonzero in only finitely many internal degrees. The proof is by induction on the number of pairs $\left( i,j \right)$ such that $\H^i(\cM)_j \neq 0$. 
  
  If $\cM$ is acyclic, then it is in $\left< \cO_X(i) \right>_{i\in \ZZ}$. Now assume that the statement of the proposition is true for all $\cM$ with at most $N$ pairs $\left( i,j \right)$ such that $\H^i(\cM)_j \neq 0$. 
  
  Write $\cM$ as the complex:
  \begin{equation*}
    \dots \too \cM^{n-1} \too \cM^n \too \cM^{n+1} \too \cM^{n+2} \too \dots\text{.}
  \end{equation*}

  Let $n$ be the lowest degree such that $\H^n(\cM)\neq 0$. Then $\cM$ is quasi-isomorphic to the complex $\tau_{\geq n}\cM$ of $\cO_X$-modules
 \begin{equation*}
   \dots \too 0 \too \cok\d^{n-1}_\cM \too \cM^{n+1} \too \cM^{n+2} \too \dots\text{.}
 \end{equation*}
By our assumption on $\cR$, this complex is also an $\cR$-module. 

Denote by $\cF$ the kernel of the morphism 
$$\d^n : \cok\d^{n-1}_\cM \too \cM^{n+1}\text{.}$$
$\cF$ is a submodule of $\tau_{\geq n}\cM$ as an $\cR$-module because it is closed under the action of $\cR$ because of our assumption on $\cR$. We have $\cF \iso \H^n(\cM)$ so $\cF$ is coherent as an $\cO_X$-module.

Now let $m\in \ZZ$ be the lowest internal degree for which $\cF_m \iso \H^n(\cM)_m \neq 0$. Since $\cF$ is coherent, $\cF_m$ is also coherent as an $\cO_X$-module. Observe that $\cR^i_j$ acts as zero on $\cF_m$ for all $(i,j)\neq (0,0)$ so $\cF_m$, which is concentrated in degrees $(n,m)$ is also an $\cR$-module.

Since $X$ is regular and affine, we have that there is a finite free resolution of $\cF_m$. So $\cF_m$ is quasi-isomorphic to a complex of $\cO_X$-modules
\begin{equation*}
  0 \too \cO_X^{\oplus r_k}(m) \too \dots \too \cO_X^{\oplus r_2}(m) \too \cO_X^{\oplus r_1}(m) \too 0\text{.}
\end{equation*}
But since $\cR^i_j$ acts as $0$ on $\cF_m$ for all $(i,j) \neq (0,0)$ and $R^0_0=\cO_X$, $\cF_m$ is also quasi-isomorphic to this complex considered as an $\cR$-module with $\cR$ acting trivially except for the $(0,0)$ piece. Therefore $\cF_m$ represents an object in $\left< \cO_X(i) \right>_{i\in \ZZ}$. The cone of the inclusion morphism $\cF_m \to \tau_{\geq n}\cM$ has the same cohomology as $\cM$ except that the piece in degrees $\left( n,m \right)$ is zero, so by the induction assumption, this cone is also in $\left< \cO_X(i) \right>_{i\in \ZZ}$. Thus $\tau_{\geq n}\cM$ and consequently $\cM$ are in $\left< \cO_X(i)\right>_{i\in \ZZ}$. 
\end{proof}

We are going to prove that the same result is true for $\cA = \exterior^\bullet \cE \otimes \cO_X[t]$. To do this, we will regrade $\cA$. Let $\cR$ be the symmetric algebra:
\begin{equation*}
  \cR = \sheafsym(\dots\to 0 \to \cE\dual \xto{-s\dual} t\cO_X \to 0 \to \dots)\text{,}
\end{equation*}
where $\cE\dual$ are in cohomological degree $-1$ and internal degree $-1$  and $t$ is in cohomological degree $0$ and internal degree $-1$ . So $\cR$ is the symmetric algebra of the same complex of vector bundles as the one for $\cA$, except that the complex is shifted twice to the left. 
We define a functor 
$$\mu : \Dgr(\cA) \to \Dgr(\cR)\text{,}$$
by $\mu(\cM)^i_j = \cM^{i-2j}_j$. It is clear that this functor, and its obvious inverse respect cones and shifts. So $\mu$ an equivalence of triangulated categories. It is also clear that the functors restrict to the subcategories $\Dbgr(\cA)$ and $\Dbgr(\cR)$. The functor $\mu$ does not respect internal degree shifts, but $\mu(\cO_X(k))=\cO_X[-2k](k)$  so we still have that $\mu$ takes the subcategory $\left< \cO_X(i) \right>_{i\in \ZZ}$ in $\Dbgr(\cA)$ to the subcategory  $\left< \cO_X(i) \right>_{i\in \ZZ}$ in $\Dbgr(\cR)$. The inverse functor also does the same in the other direction. Thus, we can apply Lemma \ref{regradedlemma} to conclude the following lemma. 

\begin{lemma}
  If a module $\cM$ in $\Dbgr(\cA)$ is has its cohomology $\H(\cA)$ coherent when considered as an $\cO_X$-module, then it is supported on $X$. Thus, the full subcategory in $\Dbgr(\cA)$ consisting of objects which have cohomology coherent over $\cO_X$ is equal to the subcategory $\Db_X(\cA)$ of modules supported on $X$.
  \label{supportlemma}
\end{lemma}

In what follows, it will be useful to consider the following diagram, which describes $\cA$ in low degrees
\begin{equation}\tag{$\ast$}\label{adiagram}
  \begin{CD}
    \hspace{0.5cm} t\ex^2\cE\dual \hspace{0.4cm} @.  \hspace{0.5cm}t^2\cO_X  \hspace{0.5cm}     @.  \hspace{0.5cm} \hspace{0.5cm}     @.   \hspace{0.5cm}    \,    \hspace{0.5cm}  @.   \hspace{0.5cm}  \hspace{0.5cm}   @.  \color{gray}{h=4} \\
    @AAA @AAA @. @. @. @. \\
    \hspace{0.5cm}\ex^3\cE\dual  \hspace{0.5cm}  @. \hspace{0.5cm} t\cE\dual  \hspace{0.5cm} @.             @.               @.             @.  \color{gray}{h=3}\\
    @. @AAA @. @. @. @. \\
    @.  \hspace{0.5cm}  \ex^2\cE\dual  \hspace{0.5cm} @.  \hspace{0.5cm}t\cO_X     \hspace{0.5cm}       @.               @.             @. \color{gray}{h=2}\\
    @. @. @AAA @. @. @. \\
    @.               @. \hspace{0.5cm}\cE\dual  \hspace{0.5cm} @.               @.             @.  \color{gray}{h=1}\\
    @. @. @. @. @. @. \\
    @.               @.             @. \,\,\,  \hspace{0.5cm} \cO_X \,\,\,  @.             @. \color{gray}{h=0}\\
    \color{gray}{i=-3}        @.    \color{gray}{i=-2}           @.      \color{gray}{i=-1}      @.  \,\,\,\,\,   \color{gray}{i=0}          @.             @. \\
  \end{CD}
\end{equation}
\vspace{0.3cm}

We are now ready to proceed with the proof of the equivalence 
 \begin{equation*}
   \Dbgr(\cA)/\Db_{X}(\cA) \iso \Dbgr(\cA[t^{-1}])\text{.}
  \end{equation*}

\begin{proof}[Proof of proposition \ref{propontheway}] 
  Consider the inclusion morphism $\phi: \cA \to \cA[t^{-1}]$ and the induced functor: 
\begin{equation*}
  \begin{array}{rcl}
    \varphi_*: \Dgr(\cA) & \too & \Dgr(\cA[t^{-1}]) \\
    \cM & \longmapsto & \cA[t^{-1}] \otimes_{\cA} \cM  \text{.}
  \end{array}
\end{equation*}
We first need to see that this functor is well-defined. We do not need to use the derived tensor product because we can identify 
\begin{equation}\label{identification}
  \varphi_* (\cM) = \cA[t^{-1}] \otimes_\cA \cM  \iso \cO_X[t,t^{-1}]\otimes_{\cO_X[t]} \cM \text{,}
\end{equation}
where the differential is $\d(t^k \otimes m) = t^k \otimes \d_\cM(m)$, linear in the first factor. Note that a section 
\begin{equation*}
  \sum_i t^i \otimes m_i
\end{equation*}
of $\cO_X[t,t^{-1}]\otimes_{\cO_X[t]} \cM$  can always be written in the form $t^k \otimes m$ by pulling out $t$'s to equalize the powers of t on the left to the lowest power of $t$ appearing in the sum. To, show that $\varphi_*$, is well-defined, we are going to show that it takes acyclic modules to acyclic modules. For this, we want to see that
\begin{equation}
  \H(\varphi_*\cM) = \H(\cO_X[t,t^{-1}]\otimes_{\cO_X[t]} \cM) \iso \cO_X[t,t^{-1}] \otimes_{\cO_X[t]}\H(\cM) \text{.}
  \label{cohomsame}
\end{equation}
We are going to show this directly. Consider the morphism of $\cO_X$-modules 
\begin{equation*}
   \alpha: \H(\cO_X[t,t^{-1}]\otimes_{\cO_X[t]} \cM) \to  \cO_X[t,t^{-1}] \otimes_{\cO_X[t]}\H(\cM) 
\end{equation*}
that sends a section represented by $t^k \otimes m$ in the kernel of the differential $\d_{\varphi_*\cM}$ to $t^k \otimes m$ where $m$ is considered as a section of the cohomology $\H(\cM)$. We need first to see that $\alpha$ takes a section represented by $t^k \otimes m$ in the kernel of the differential to a section of the form $t^{k'} \otimes m'$ where $m'$ is in the kernel of $\d_\cM$. Indeed, if $t^k \otimes m$ is such a section, then $\d(t^k\otimes m) = t^k \otimes \d_\cM(m) = 0$, which means that there is a $p>0$ such that $t^p \d_\cM(m) = 0$ in $\cM$. But $t^k \otimes m = t^{k-p} \otimes t^p m$ with $\d_\cM(t^p m ) = 0$.
The map $\alpha$ is well-defined since $\d(t^k \otimes m) = t^k \otimes d(m)$ so an element in the image of the differential is taken to zero. It is clear that $\alpha$ has a well-defined inverse $\beta$ which takes a section represented by $t^k\otimes m$ with $m \in \ker\d_\cM$ to $t^k \otimes m$ considered as a section of the kernel of the differential of $\H(\cO_X[t,t^{-1}]\otimes_{\cO_X[t]} \cM)$.

So if $\cM$ has $\H(\cM)=0$, then $\H(\varphi_*\cM)=0$. Therefore the functor $\varphi_*$ is well-defined. 

Consider the functor in the opposite direction:
\begin{equation*}
  \begin{array}{rcl}
     	\varphi^*: \Dgr(\cA[t^{-1}]) & \too & \Dgr(\cA)\\
	\cN & \longmapsto & \cN_{\leq 0} \text{,}
\end{array}
\end{equation*}
where $\cN_{\leq 0}$ denotes the part of $\cN$ with non-positive internal grading. This functor is well-defined as well since it takes acyclic modules to acyclic modules. 

We claim that these functors induce the equivalence of the proposition. First, $\varphi_*$ descends to a functor $\bar{\varphi_*}$ from the quotient $\Dgr(\cA)/\Db_{X}(\cA)$ since if we take an object locally in $\left< \cO_X(i) \right>_{i\in \ZZ}$, its cohomology is $t$-torsion so and by using (\ref{cohomsame}), we can see that its image is 0 under $\bar{\varphi_*}$. Second,  these functors take elements of $\Dbgr(\cA)$ into elements of $\Dbgr(\cA[t^{-1}])$ and vice versa. Third, the composition of these functors is isomorphic to the identity functor; which is what we show next.

Consider the natural transformation
\begin{equation*}
  \operatorname{Id} \too \varphi^* \circ \bar\varphi_*
\end{equation*}
which is given by the the morphisms 
\begin{equation*}
  \cM \too (\cA[t,t^{-1}] \otimes_\cA \cM)_{\leq 0} \iso ( \cO_X[t,t^{-1}] \otimes_{\cO_X[t]}  \cM)_{\leq 0}
\end{equation*}
given for each $\cM \in \Dbgr(A)$
 by taking sections $m$ to $ 1\otimes m$ if their internal degrees are non-positive and to 0 if their internal degrees are positive. Let $\cJ$ be the cone of this morphism. We have the long exact sequence of sheaves of $\cO_X$-modules in cohomology:
\begin{equation*}
  \dots\too \H^{i}(\cM) \too \H^i(\varphi^* \circ \bar\varphi_* \cM) \too \H^{i}(\cJ) \too \H^{i+1}(\cM) \too \H^{i+1}(\varphi^* \circ \bar\varphi_* \cM) \too \dots \text{.}
\end{equation*}
So we have the short exact sequence
\begin{equation*}
  0 \too \cok(\alpha_i) \too \H^i(\cJ) \too \ker(\alpha_{i+1})\too 0\text{,}
\end{equation*}
where 
\begin{equation*}
  \alpha_i: \H^i(\cM) \to \H^i(( \cO_X[t,t^{-1}] \otimes_{\cO_X[t]}  \cM)_{\leq 0})
\end{equation*}
is the induced map on cohomology. By using (\ref{cohomsame}), the coherence of $\H(\cM)$ over $\H(\cA)$ and the fact that below degree $-r$, all sections of $\H(\cA)$ in internal degree $-j\leq-r$ are of the form $t^j a$ for sections $a$ of internal degree 0, one can show that $\ker(\alpha_{i+1})$ and $\cok(\alpha_i)$ are coherent over $\cO_X$. So by the short exact sequence above, $\H(\cJ)$ is coherent over $\cO_X$. Therefore by Lemma \ref{supportlemma}, the cone $\cJ$ a module supported on $X$. Therefore this natural transformation is an isomorphism of functors.  

On the other hand, consider the natural transformation
\begin{equation*}
  \bar\varphi_* \circ \varphi^* \too \Id
\end{equation*}
given by the morphisms 
\begin{equation*}
  \cA[t^{-1}]\otimes_\cA (\cN)_{\leq 0} \too \cN
\end{equation*}
given for each $\cN \in \Dbgr(\cA[t^{-1}])$, by taking sections $a\otimes n$ to $a n$. For each $\cN$, this morphism is clearly surjective. It is also injective since if we consider a section $t^k \otimes n$ (by the isomorphism (\ref{identification})) whose image is $0$, then $t^k  n = 0$, so $n = t^{-k} t^k  n = 0$. Hence this natural transformation is also an isomorphism; which completes the proof of the equivalence.
\end{proof}

Before, we move to the last steps in the proof of our equivalence,
recall that when $Y$ is expressed as the zero locus of the regular section $s \in \H^0(X,\cE)$, the Koszul resolution of $\cO_Y$ is given by
\begin{equation*}\label{koszulresolution}
  0 \to \exterior^r \cE\dual \to \dots \to \exterior^2 \cE\dual \to \cE\dual \to \cO_X \to \cO_Y \to 0\text{,}
\end{equation*}
where the differential is given by $\d(f) = f(s)$, and is extended by the Leibnitz rule. If we denote this resolution by $\cK$, then we have $\cK^{1}=0$, $\cK^{0} = \cO_X$, $\cK^{-1} = \cE\dual$, $\cK^{-2} = \exterior^2 \cE\dual$ and so on. Here, all components are in internal degree $0$. 

We now consider the structure of $\cA$ in more detail. Observe from the diagram (\ref{adiagram}) on page \pageref{adiagram} or from direct computation that in each internal degree, there is a bluntly truncated and shifted copy of the Koszul resolution of $\cO_Y$. On the other hand, $\cA[t^{-1}]$ has, in each internal degree, a shifted copy of the full Koszul resolution of $\cO_Y$, since we now have the rest of the Koszul grading accompanied by negative powers of $t$. So the cohomology of $\cA[t^{-1}]$ is 
\begin{equation*}
  \H(\cA[t^{-1}]) \iso \cO_Y[t,t^{-1}]\text{.}
\end{equation*}
Since in each internal degree $\cA[t^{-1}]$ is acyclic except at $t^k\cO_X$, the morphism of sheaves that takes $t^k \cO_X$ to $t^k\cO_Y$ by restriction, and everything else to $0$, is a morphism of sheaves of dg algebras. So we have a quasi-isomorphism
\begin{equation*}
  \psi : \cA[t^{-1}] \to \H(\cA[t^{-1}])  \iso \cO_Y[t,t^{-1}]\text{.}
\end{equation*}
By the same regrading trick we used above, we can apply Proposition \ref{transfer}. Thus we have arrived at:

\begin{proposition}
  There is a quasi-isomorphism between $\cA[t^{-1}]$ and its cohomology algebra
     $ \H(\cA[t^{-1}])\iso \cO_Y[t,t^{-1}]$. This quasi-isomorphism induces an equivalence
  \begin{equation*}
    \Dbgr(\cA[t^{-1}]) \iso \Dbgr(\cO_Y[t,t^{-1}])  \text{.}
  \end{equation*}
\end{proposition}

Since the category of graded modules over $\cO_Y[t,t^{-1}]$ is equivalent to the category of $k^\times$-equivariant modules over $Y \times (\mathbb{A}^1 \backslash \{0\})$, we have
  \begin{equation*}
    \Dbgr(\cO_Y[t,t^{-1}]) \iso \Db(\cO_Y)\text{.}
  \end{equation*}

Combining our results above gives the following chain of equivalences:
\begin{equation*}
  \Dsg^{k^\times}(Z) \iso \frac{\Dbgr(\cB)}{\t{Perf}{\cB}} \iso \frac{\cDbgr(\cA)^\t{op}}{\Db_{X}(\cA)^\t{op}} \iso \cDbgr(\cA[t^{-1}])^\t{op} \iso \Dbgr(\cO_Y[t,t^{-1}])^\t{op} \iso \Db(Y)^\t{op} \iso \Db(Y),
\end{equation*}
where the last equivalence is given by the functor $\sheafrhom(\bullet, \cO_Y)$, which is an equivalence of the bounded derived category of $Y$ and its opposite category because $Y$ is a local complete intersection in the regular variety $X$ and is therefore Gorenstein. 
This gives us our main theorem

\begin{theorem}\label{maintheorem}
  There is a an equivalence of triangulated categories:
  \begin{equation*}
    \Dsg^{k^\times}(Z) = \Db(Y)\text{.}
  \end{equation*}
  \label{finalequiv}
\end{theorem}

\begin{remark}
  When $s$ is not a regular section, we get the same chain of equivalences except that the equivalence $\Dbgr(\cO_Y[t,t^{-1}]) \iso \Db(Y)$ does not hold since the Koszul complex $\cK$ above is not a resolution anymore. So in this case we get an equivalence between $\Dsg^{k^\times}(Z)$ and $\Db(\cK)$. 
\end{remark}

The results of V. Lunts and D. Orlov in \cite{uniquenessofenhancements} give that two dg categories which have their homotopy categories equivalent to the same bounded derived category of coherent sheaves on a variety with enough locally free sheaves are quasi-equivalent --- by which we mean that they are equivalent in $\t{Ho}(\t{dg-Cat})$, the localization of the category of dg categories by quasi-equivalences. Theorem \ref{finalequiv} thus gives
\begin{corollary}
  The dg categories $\cD^{{b}}(Y)$ and $\cD^{k^\times}_{{sg}}(Z)$ are quasi-equivalent.
\end{corollary}

\bibliographystyle{amsalpha}
\bibliography{bibliography}

\end{document}